\documentclass[11pt,twoside,english]{amsart}
\normalsize
\advance\oddsidemargin by -1.0cm
\advance\evensidemargin by -1.0cm
\advance\topmargin by -1.0cm 
\usepackage{amssymb}
\usepackage{babel}
\usepackage{amsmath}
\usepackage{amscd}   
\usepackage{epsfig}  
\usepackage{rotating}
\usepackage{psfrag}

\theoremstyle{plain}

\newtheorem{thm}{Theorem}[section]

\theoremstyle{definition}
\newtheorem{exa}{Example}[section]

\newtheorem{dfn}{Definition}[section]

\newcommand{\bdm}{\begin{displaymath}}
\newcommand{\edm}{\end{displaymath}}
\newcommand{\be}{\begin{equation}}
\newcommand{\ee}{\end{equation}}
\newcommand{\ba}[1]{\begin{array}{#1}}
\newcommand{\ea}{\end{array}}

\newcommand{\btab}{\begin{tabular}}
\newcommand{\etab}{\end{tabular}}

\newcommand{\ra}{\rightarrow}

\begin{document}
\def\haken{\mathbin{\hbox to 6pt{%
                 \vrule height0.4pt width5pt depth0pt
                 \kern-.4pt
                 \vrule height6pt width0.4pt depth0pt\hss}}}
    \let \hook\intprod
\setcounter{equation}{0}
%
%
\thispagestyle{empty}
%
\date{\today}
\title[isoparametric hypersurface]{Isoparametric hypersurfaces and their applications to Special Geometries}
%
%
%
\author{firouz khezri}
\address{\hspace{-5mm} 
{\normalfont\ttfamily khezri@mathematik.hu-berlin.de}\newline
Institut f\"ur Mathematik \newline
Humboldt-Universit\"at zu Berlin\newline
Unter den Linden 6\newline
Sitz: John-von-Neumann-Haus, Adlershof\newline
D-10099 Berlin, Germany}
%
\thanks{Supported by Berlin Mathematical School} 

\maketitle
\pagestyle{headings}
%
%
%
\section{ Motivation}
Let $\Upsilon$ be an object (e.g tensor) whose isotropy group under the action of $\mathrm SO({n_k})$ is $H_{k}$$\subseteq\mathrm SO({n_k})$.
Consider $H_k$ structures $(N^{n_k},g,\Upsilon)$  with Levi-Civita connection $\overline{\Gamma}\in\mathfrak{so}{(n_k)}\otimes\mathbb{R}^{n_k} $  uniquely decomposable according to: 
\bdm
 \overline{\Gamma}=\Gamma+\frac{1}{2}T, 
\edm
where $\Gamma\in\mathfrak{h}_{k}\otimes\mathbb{R}^{n_k}$ 
and $T\in\Lambda^3(\mathbb{R}^{n_k})$. 
A $\mathfrak{h}_k$-valued connection $\Gamma$ of $H_k$ structure is called characteristic connection.\vspace{6 mm}

We could ask the following questions:

\begin{enumerate}
\item   In which dimensions ${n_k}$ of $N$ does such a decomposition exist?

\item   What is $\Upsilon$ which reduces $\mathrm SO({n_k})$ to $H_k$ for them?

\item   What are the possible isotropy groups $H_k$ of $SO({n_k})$?
\end{enumerate}
\vspace{2mm}

Moreover, if  $T\in\Lambda^3(\mathbb{R}^{n_k})$ was identically zero, then since $\overline{\Gamma}=\Gamma $,
the holonomy group of $(N^{n_k},g,\Upsilon)$ would be reduced to $H_k\subseteq\mathrm SO({n_k})$.
All irreducible compact Riemannian manifold $(N^{n_k},g,\Upsilon)$ with reduced holonomy group are classified by Berger; In this note we will focus on irreducible symmetric spaces $G/H$ from Cartan's list, with the holonomy group $H_k\subseteq \mathrm{SO}({n_k})$.
The simplest irreducible symmetric space in Cartan's list is $N^5=SU(3)/SO(3)$.\\

Pawel Nurowski \cite{Nur} studied $5$-dimensional manifolds $(N,g,\Upsilon)$ and states that such $\Upsilon$ whose isotropy group under the action of $\mathrm{SO}(5)$ is the irreducible  $\mathrm{SO}(3)$ is determined by the following conditions:

\begin{enumerate}
\item $\Upsilon_ {ijk}=\Upsilon _{(ijk)}$  \hspace{5mm}    (totally symmetric)

\item $\Upsilon_ {ijj}=0$ \hspace{13mm}                  (trace-free)

\item $\Upsilon_{ijk}\Upsilon_{lmi} + \Upsilon_{lji}\Upsilon_{kmi} + \Upsilon_{kli}\Upsilon_{jmi}=g_{jk}g_{lm} + g_{lj}g_{km} +g_{kl}g_{jm}$.
\end{enumerate}

The tensor $\Upsilon$ can be defined by means of a homogeneous polynomial of degree 3:
\begin{eqnarray*}
\Upsilon_{ijk}x_{i}x_{j}x_{k} &=& \frac{1}{2} \det \left(\begin{array}{ccc}
x_{5}-\sqrt{3}x_{4} & \sqrt{3}x_{3} & \sqrt{3}x_2 \\
\sqrt{3}x_3 & x_{5}+\sqrt{3}x_{4} &\sqrt{3}x_1\\
\sqrt{3}x_2 & \sqrt{3}x_1 &-2x_5 \\ 
\end{array}\right)\\
&=& x^3_5+\frac{3}{2}x_5(x^2_1+x^2_2)-3x_5(x^2_3+x^2_4)+\frac{3\sqrt{3}}{2}x_4(x^2_1-x^2_2)+3\sqrt{3}x_1x_2x_3.
\end{eqnarray*}
Then he uses the results of Cartan's work on isoparametric hypersurfaces and claims that:

\begin{thm}
The only dimensions $n_k$ in which conditions $(1)-(3)$ admit a solution for $\Upsilon_{ijk}$ are $n_k=3k+2$, where $k=1,2,4,8$.
\end{thm}
\begin{thm}
In dimensions $n_k=5,7,13$ and $26$ tensor $\Upsilon$ reduces the $\mathbf GL(n_k,\mathbb R)$ via $\mathrm O(n_k)$ to a subgroup $\mathrm H_k$, where:
\begin{itemize}
\item  For $n_1=5$ the group $H_1$ is the irreducible $\mathrm SO(3)$ in $\mathrm SO(5)$; The torsionless compact model is $SU(3)/SO(3)$
\item  For $n_2=8$ the group $H_2$ is the irreducible $\mathrm SU(3)$ in $\mathrm SO(8)$; The torsionless compact model is $SU(3)/SU(3)/SU(3)$
\item  For $n_4=14$ the group $H_{4}$ is the irreducible $\mathrm Sp(3)$ in $\mathrm SO(14)$; The torsionless compact model is  $SU(6)/Sp(3) $
\item  For $n_8=26$ the group $H_{8}$ is the irreducible $\mathrm F_4(3)$ in $\mathrm SO(26)$; The torsionless compact model is   $E_6/F_4$. 
\end{itemize}
\end{thm}

\vspace{2mm}
Pawel Nurowski does not give complete proof for his claim, in this note we show  how this claim can be reduced to a classification result on principle orbits of rank $2$ symmetric spaces by Hsiang and Lawson\cite{HL}; For this purpose we will give a brief introduction to the theory of isoparametric hypersurfaces in section 2, then we will see the argument of the claimed statement in section 3. Finally we end  this note with inhomogeneous isoparametric hypersurfaces whose construction is nicely related to  Clifford algebras, and we will collect open problems in isoparametric hypersurfaces at the end of this note. 

\vspace{10mm}
I would like to thank Prof. Ilka Agricola for having given me this nice topic to investigate. She was always open to my arising questions and guiding me patiently.

\vspace{2mm}
 My cordial thanks go to Prof. Dirk Ferus who helped me so much to understand the concept of isoparametric hypersurface. Specially I am very grateful for his generous help to provide me  nice articles on isoparametric hypersurface, and his helpful lecture notes on the topic.

\vspace{2mm}
Last but not least, I would like to thank the Berlin Mathematical School and the One-Stop-Office for the financial and 
un-financial support by providing the atmosphere conducing to fruitful research.

\section{ Basic  Properties of isoparametric hypersurfaces}
%

\begin{dfn}
A family of isoparametric  hypersurface of a Riemannian manifold $\tilde{M}^{n}$ is a collection of  level sets of a non-constant, real valued and smooth function $f$ defined on an open and connected subset of $\tilde{M}^{n}$ such that the gradient and Laplacian of $f$ satisfy

\begin{equation}
\label{iso}
|grad(f)|^2=F_1(f)  ,\hspace{5mm}  \Delta(f)=F_2(f),
\end{equation}
where  $F_1,F_2$ are smooth functions $:\mathbb{R}\ra\mathbb{R}$.
Such $f$ with this properties is called  isoparametric function.
\end{dfn}
Hence $|grad(f)|^2$ and $\Delta (f)$ are functions of $f$ and these quantities are constant on each level set of $f$. Constancy of the gradient along the level set of a function $f$ implies that the families $\{f^{-1}(c)\}$ are parallel. We show that constancy of the Laplacian together with constancy of the gradient implies that the mean curvature of the hypersurface is constant, for this purpose assume that  $f$ is as above and $M^{n-1}_c=\{x\in\tilde{ M}^{n} | f(x)=c\}$ is a nonsingular $(grad(f)\neq 0)$ hypersurface. Under this assumption
\begin{equation}
\xi=\frac{grad(f)}{|grad(f)|},
\end{equation}
is a unit normal vector field to the hypersurface $M^{n-1}_c$. The second fundamental form $h$ for $M^{n-1}_c$ is given by
\begin{equation}
h(X,Y)=\frac{-H_f(X,Y)}{|grad(f)|}, 
\end{equation}
where $H_f$ denotes the Hessian of the function $f$ on $\tilde{M}^n$. The mean curvature (=trace of $h$) 
is given by 

\begin{equation}
\label{mean}
\frac{(grad(f))(|grad(f)|)- (|grad(f)|)\Delta_f}{|grad(f)|},
\end{equation}
but (\ref{iso}) implies that  (\ref{mean}) is constant, that is, each  hypersurface has constant mean curvature. If we assume that the ambient space $\tilde M^n$ is a manifold with constant sectional curvature then  hypersurface $M^{n-1}$  has constant mean curvature if and only if its all principal curvatures are constant\cite{Nom2}. Actually this fact is a starting point in order to do our  main job, so let to state this as a new definition.
\begin{dfn}
Let $M^{n-1}$ be an  immersed hypersurface in a simply connected, complete Riemannian manifold $\tilde{M}^n$  with constant sectional curvature, in particular let $\tilde {M}^n$ be  $\mathbb{R}^n$, $\mathrm{S}^n$ or $\mathrm{H}^n$, then we call $M^{n-1}$  an isoparametric hypersurface if all its  principal curvatures are constant.
\end{dfn}

Notice that in general it is not true that isoparametric hypersurfaces in any Riemannian manifold have constant principal curvatures. Examples have been found in $P^n(\mathbb{C})$ by Wang \cite{Wang1}.

Now we can also give a more intuitive geometric definition of isoparametric hypersurface.

\begin{dfn}
Let $\tilde M^n$ be a real space-form namely $\mathbb{R}^n$, $\mathrm{H}^n$ and $\mathrm{S}^n$ then an isoparametric family of hypersurfaces  is a family of parallel hypersurfaces $g_t:M^{n-1} \rightarrow \tilde{M}^n$ obtained from a hypersurface $g:M^{n-1} \rightarrow \tilde{M}^n$ with constant principal curvatures
\end{dfn}

For $\tilde{M}^n = {\mathbb {R}}^n$, the formula for $g_t$ is
\begin{gather}
\label{parallel-E}
g_t(x) = g(x) + t \xi(x),
\end{gather}
and for $\tilde{M}^n = S^n$, the formula for $g_t$ is{\samepage
\begin{gather}
\label{parallel-S}
g_t(x) = \cos t \; g(x) + \sin t \; \xi(x).
\end{gather}
There is a similar formula in hyperbolic space (see, for example, \cite{CecToh}).}

In 1938 Segre \cite{Seg} showed that an isoparametric hypersurface in $\mathbb{R}^n$ can have at most two distinct principal curvatures, and it must be an open subset of a hyperplane, hypersphere or spherical cylinder $S^k\times\mathbb{R}^{n-k}$.
A similar result holds for $H^n$. Thus isoparametric hypersurfaces in this two cases  are completely classified.
However conditions in $S^n$  do not lead to such  strong restrictions, and Cartan found far more examples of isoparametric hypersurfaces in $S^n$ that are all geometrically interesting. Nevertheless till now there is no  complete classification of isoparametric hypersurfaces of $S^n$. From now on we 
will only focus on isoparametric hypersurfaces on $S^n$ and all discussion will be on spherical case. But before starting to give a classification we need some preparation. 

Locally, for suf\/f\/iciently small values of $t$, the map $g_t$ (\ref{parallel-S}) is also an immersed hypersurface.  However, the map
$g_t$ may develop singularities at the focal points of the original hypersurface $g(M^{n-1})$. More precisely: 

\begin{dfn}
Let $M^{n-1}$ be an immersed hypersurface in $S^n$ and let $\xi$ be a unit normal vector field. A point $y_0\in S^n$ is called focal point of $(M^{n-1},x_0)$ where $x_0\in M^{n-1}$, if
\begin{equation}
 y_0=\cos t_0\;x_0+\sin t_0\;x_0,
\end{equation} 
for some $t_0$ and if the differential of the mapping
\begin{equation}
 (x,t)\rightarrow \cos t\;x+\sin t\;x,
\end{equation}
is singular at $(x_0,t_0)$. A point $y_0$ is called a focal point of $M^{n-1}$ if it is focal point of $(M^{n-1},x_0)$ for some $x_0\in M^{n-1}$.
\end{dfn}

We see that
\begin{equation}
y_0=\cos t_0\;x_0+\sin t_0\;x_0,
\end{equation} 
is a focal point of $(M^{n-1},x_0)$ if and only if $\cot t_0$
is one of the principal curvatures.

Now assume that $M^{n-1}$ has $p$ distinct constant principal curvatures $a_1=\cot\theta_1 ,...,a_p=\cot\theta_p$ of multiplicities $m_1$,...,$m_p$. For each $i$, $1\leq i\leq p$, let 
\begin{equation}
T_i(x)=\{X\in T_xM^{n-1}| AX=a_iX; \hspace{2mm} x\in M^{n-1}\}, 
\end{equation}
where $A$ is the shape operator. The distribution $T_i:x\rightarrow T_i(x)$ of dimension $m_i$ is integrable and any maximal integral manifold $M_i(x)$ of $T_i(x)$ through $x$ is totally geodesic in $M^{n-1}$ and umbilical as a submanifold of $S^n$. Moreover if $M^n$ is connected and compact, then $M_i(x)$ is a $m_i$-dimensional small sphere of $S^n$.

We define a differential mapping $g_i:M^{n-1}\rightarrow S^n $ by
\begin{equation}
g_i(x)=\cos\theta_ix+\sin\theta_i\xi x ,
\end{equation}
where $\theta_i$=arcot$a_i$. The differential of $g_i$ at $x$ has $T_i(x)$ as the null space and injective on the subspace $\sum_{k\neq i}T_k(x)$ of $T_xM$. It follows that for each point $x$ there is a local coordinate system $u^1,...,u^n$ with origin $x$ such that $g_i$ is a one-to-one immersion of the slice $u^1=...=u^{m_i}$ into $S^n$, the image being contained in $g_i(M^{n-1})$. In this way $g_i(M^{n-1})$ is a submanifold of dimension $(n-m_i-1)$ in a neighborhood of $g_i(x)$. Geometrically, $g_i$ maps $M_i(x)$ into one single point $g_i(x)$ and $g_i(M^{n-1})$ collapses by dimension $m_i$.
So $g_i(M^{n-1})$ is  an $(n-m_i-1)$-dimensional submanifold which is called focal submanifold of $M^{n-1}$.

Let $M^{n-1}_t$ be a family of the parallel family arising from the $M^{n-1}=M^{n-1}_0$. We want  to see how the principal curvatures of the original hypersurface $M^{n-1}_0$ are related to the principal curvatures of the immersed hypersurfaces $M^{n-1}_t$. If $a_i=\cot\theta_i$ is one of the principal curvatures of $M^{n-1}$ at $x\in M^{n-1}$, then the corresponding principal curvature of   $M^{n-1}_t$ at the point $x_t$ is $\cot(\theta-t)$. The geodesic distance from $x_0$ to $g_i(x_0)$ is $\theta_i$ and the geodesic distance from $x_t$ to $g_i(x_0)$ is $\theta-t$. This means that the focal submanifold of $M^{n-1}_t$ corresponding to the principal curvature $\cot(\theta-t)$ coincides with focal submanifold of $M^{n-1}$ corresponding to $\cot\theta_i$. Hence all isoparametric hypersurfaces in the family have the same focal submanifold. Moreover we will see later that each family of isoparametric hypersurface has two focal submanifolds.

 M\"{u}nzner published two preprints that
greatly extended Cartan's work and have served as the basis for much of the research in the f\/ield since that
time.  The preprints were eventually published as papers \cite{Mu,Mu2} in 1980--1981.

In the f\/irst paper \cite{Mu} ,
M\"{u}nzner began with a geometric study of the focal submanifolds of an isoparametric
hypersurface $g:M^{n-1} \rightarrow S^n$ with $p$ distinct principal curvatures.  Using the fact that each focal
submanifold of $g$ is obtained as a parallel map $g_t$, where $\cot t$ is a principal curvature of $g(M^{n-1})$,
M\"{u}nzner computed a formula for the shape operator of the focal submanifold $g_t(M^{n-1})$ in terms of the
shape operator of $g(M^{n-1})$ itself. Then by using a symmetry argument, he proved that if the principal curvatures
of $g(M^{n-1})$ are written as $\cot \theta_k, 0 < \theta_1 < \cdots < \theta_p < \pi$, with multiplicities $m_k$, then
\begin{gather}
\label{prin-curv}
\theta_k = \theta_1 + \frac{(k-1)}{p} \pi, \qquad 1 \leq k \leq p,
\end{gather}
and the multiplicities satisfy $m_k = m_{k+2}$ (subscripts mod $p$).  Thus, if $p$ is odd, all
of the multiplicities must be equal, and if $p$ is even, there are at most two distinct multiplicities.

If $\cot t$ is not a principal curvature of $g$, then the map $g_t$ in equation (\ref{parallel-S}) is also an
isoparametric hypersurface with $p$ distinct principal curvatures $\cot (\theta_1 - t),\ldots,\cot (\theta_p -t)$.
If $t = \theta_k$ (mod~$\pi$),
then the map $g_t$ is constant along each
leaf of the $m_k$-dimensional principal foliation~$T_k$, and the image of $g_t$ is a smooth focal submanifold of $g$ of
codimension $m_k +1$ in $S^n$. All of the hypersurfaces $g_t$ in a family of parallel isoparametric hypersurfaces
have the same focal submanifolds.

In a crucial step in the theory,
M\"{u}nzner then showed that $g(M^{n-1})$ and its parallel hypersurfaces and focal submanifolds are each contained in
a level set of a homogeneous polynomial~$F$ of degree $p$ satisfying the following
{\em Cartan--M\"{u}nzner differential equations}
involving the Euclidean dif\/ferential operators ${\rm grad}\, F$ and Laplacian $\triangle F$ on ${\mathbb {R}}^{n+1}$,
\begin{gather}
\label{eq:C-M}
|{\rm grad}\,F|^2  =   p^2 r^{2p-2}, \qquad r = |x|, \\
 \triangle F  =  c r^{p-2}, \qquad c = p^2 (m_2-m_1)/2, \nonumber
\end{gather}
where $m_1$, $m_2$ are the two (possibly equal) multiplicities of the principal curvatures on $f(M^{n-1})$.

Conversely, the level sets of the restriction $F|_{S^n}$ of a function $F$ satisfying equations (\ref{eq:C-M}) constitute
an isoparametric family of
hypersurfaces and their focal submanifolds, and $F$ is called the
{\em Cartan--M\"{u}nzner polynomial} associated to this family.
Furthermore, M\"{u}nzner showed that the level sets of $F$ are connected, and thus any
connected isoparametric hypersurface in $S^n$ lies in a unique compact, connected isoparametric hypersurface
obtained by taking the whole level set.

The values of the restriction $F|_{S^n}$
range between $-1$ and $+1$.  For $-1 < t < 1$, the level set $M_t = (F|_{S^n})^{-1}(t)$ is an isoparametric hypersurface,
while $M_{+} = (F|_{S^n})^{-1}(1)$ and $M_{-} = (F|_{S^n})^{-1}(-1)$ are focal submanifolds.  Thus, there are exactly
two focal submanifolds for the isoparametric family, regardless of the number $p$ of distinct principal curvatures.
Each principal curvature $\cot \theta_k$, $1 \leq k \leq p$, gives rise to two antipodal focal points corresponding
to the values $t = \theta_k$ and $t = \theta_k + \pi$ in equation (\ref{prin-curv}).  

Now we can start our classification, let $p$ be a number of distinct principal curvatures of isoparametric hypersurface in $S^n$ then Cartan himself could comletely classified all isoparametric heypersurfaces with $p\leq 3$ and gave examples with $p=4$. In the next coming paragraph we explain his work.

In the case $p=1$,  $M^{n-1}$ is an open subset of a great or small hypersphere in $S^n$.
 
If $p=2$, then $M^{n-1}$ must be a standard product of two spheres,
\begin{gather}
\label{product}
S^a (r) \times S^b (s) \subset S^n (1) \subset {\mathbb {R}}^{a+1}
\times {\mathbb {R}}^{b+1} = {\mathbb {R}}^{n+1}, \qquad r^2 + s^2 = 1,
\end{gather}
where $n = a+b+1$.

The case $p=3$ is one of the interesting case, here we see why we have such a restricted dimensions   in  Theorem 1.1. 
Cartan \cite{Car3} showed that this case can occur  only in a certain dimensions of sphere, in particular in dimensions $4,7,13$ and $25$ of sphere. Isoparametric hypersurfaces with $p=3$ must be a tube of
constant radius over a standard embedding of a projective
plane ${\bf P}^2(\mathbb F)$ into $S^{n}$ (see, for example, \cite[pp.~296--299]{CR}),
where ${\mathbb F}$ is the division algebra
${\mathbb {R}}$, ${\mathbb C}$, ${\mathbb H}$ (quaternions),
${\mathbb O}$ (Cayley numbers). Then he obtained isoparametric function whose level sets define isoparametric hypersurface family. Cartan showed that any isoparametric family with $p$
distinct principal curvatures of the same multiplicity can be def\/ined by an equation of the form
$F = \cos pt$ (restricted to $S^n$), where $F$ is a harmonic homogeneous polynomial of degree $p$ on
${\mathbb {R}}^{n+1}$ satisfying
\begin{gather}
\label{car-eq}
|{\rm grad}\,F|^2 = p^2 r^{2p-2},
\end{gather}
where $r = |x|$ for $x \in {\mathbb {R}}^{n+1}$, and ${\rm grad}\, F$ is the gradient of $F$ in ${\mathbb {R}}^{n+1}$. This was
a forerunner of M\"{u}nzner's general result (\ref{eq:C-M}) that every isoparametric hypersurface is algebraic, and its
def\/ining polynomial satisf\/ies certain dif\/ferential equations which generalize those that Cartan found in this special case. When $p=3$, then Cartan obtained homogeneous polynomial $F$ of degree 3
\begin{equation}
 F(X)=u^3-3uv^2
+\frac{3}{2}u^2(x\overline{x}+y\overline{y}-2z\overline{z})+\frac{3\sqrt{3}}{2}v(x\overline{x}-y\overline{y})
+\frac{3\sqrt{3}}{2}(xyz+\overline{x}\overline{y}\overline {z}),
\end{equation}
where for,
\begin{itemize}
\item $n=4$  $\hspace{2.7mm}$   $x,y,z\in\mathbb{R}$
\item $n=7$  $\hspace{2.7mm}$   $x,y,z\in\mathbb{C}$
\item $n=13$ $\hspace{1mm}$     $x,y,z\in\mathbb{H}$
\item $n=25$ $\hspace{1mm}$   $x,y,z\in\mathbb{O}$.
\end{itemize}

Unfortunately we could not understand (for language and notation problem of old text) how Cartan argued for dimension problem in this case. However Prof. Dirk Ferus \cite{DF1} gave us one alternative argument based on an unpublished paper of Prof. H. Karcher which we are going to explain it in next paragraph. We also refer to one another argument in \cite{CO} by Console and Olmos, they related the shape operator of the hypersurface to the Clifford system then by using representation theory of Clifford algebra they showed that isoparametric hypersurface of $S^n$  with three distinct principal curvatures can only occur in dimension $4,7,13$ and $25$.

The goal is to construct a certain division algebra by using the properties of isoparametric hypersurfaces with $3$ distinct principal curvatures then our argument follows from dimension of this division algebra.  Consider a manifold $M$ with covariant derivative $\nabla$. The latter extends naturally to the tensor algebra of $M$. In particular, for a (1,1)-tensor field $A$ we have
\bdm
(\nabla_XA).Y:=\nabla_X(AY)-A(\nabla_XY),
\edm

Now consider an isoparametric hypersurface $g:M^{n-1}\rightarrow S^n$ with shape operator $A$ and principal curvatures $\lambda_i$. 
Put $E_i:=$ker$(A-\lambda_iId)$, and take vector fields $X_i,Y_i,Z_i\in\Gamma(E_i)$. Then 
\bdm
(\nabla_{X_i}A).Y_j=\nabla_{X_i}(\lambda_jY_j)-A\nabla_{X_i}Y_j=(\lambda_j-A)\nabla_{X_i}Y_j\in E^{\perp}_j.
\edm
Using Codazzi's equation $(\nabla_{X}A).Y=(\nabla_YA).X$ we obtain moreover 
\begin{equation}\label{2}
(\nabla_{X_i}A).Y_j=(\lambda_i-A)\nabla_{Y_j}X_i,
\end{equation}
and therefore
\begin{equation}\label{3}
(\nabla_{X_i}A).Y_j\in\sum_{k\neq {i,j}}E_k.
\end{equation}
For $i\neq j$ the self-adjointness of $\nabla_{X_i}A$ yields
\begin{equation}\label{4}
\langle(\nabla_{X_i}A).Y_i,Z_j\rangle=\langle{Y_i},(\nabla_{X_i}A).Z_j\rangle=0,
\end{equation}
since $(\nabla_{X_i}A).Z_j\in E_i^{\perp}$. Hence, by ($\ref{3}$) and ($\ref{4}$),
\begin{equation}\label{5}
(\nabla_{X_i}A).Y_i=0,
\end{equation}
or
\bdm
A(\nabla_{X_i}Y_i)=\nabla_{X_i}(AY_i)=\lambda\nabla_{X_i}Y_i.
\edm
This shows that $E_i$ are integrable, and its integral manifolds the $\lambda_i-$leaves, are totally geodesic in $M$. If the ambient space is a sphere, the leaves are small spheres, and the idea is the reconstruction of the immersion $g$ by moving the leaves along each other.

We now restrict to the case  $p=3$, that is, M is hypersurface in $S^n$ with three distinct principal curvatures $\lambda_1,\lambda_2$ and $\lambda_3$ with multiplicities $m_1,m_2$ and $m_3$ respectively. Since the number of distinct principal curvatures is odd, all multiplicities must be equal, so let $m:=m_1=m_2=m_3$ then $n=3m+1$. Hence the isoparametric hypersurface is immersed in an $(3m+1)$-dimensional sphere, and it suffices to show that $m$ can be only $1,2,4,8$. If we choose isomorphisms $\rho:E_2\rightarrow E_1$
and $\sigma:E_3\rightarrow E_2$ then
\bdm
YZ:=\sigma\left((\nabla_{\rho(Y)}A).Z\right),
\edm
for $Y,Z\in E_2$ defines a division algebra on $E_2$ . Note that $(\nabla_{\rho(Y)}A).Z\in E_3$ for $Y,Z\in E_2$ by $(\ref{3})$.
Hence  $m=dim E_2\in\{1,2,4,8\}$ and from this we obtain that dimensions of sphere having isoparametric hypersurfaces with three distinct principal curvatures can be only $4,7,13$ and $25$.  From the classification of those algebra, we conversely get information on $\nabla A$, i.e on the initial condition for the equation of motion of $A$ along the $\lambda_i$-leaves. This allows the construction  of $g$. Thus, in case $p=3$ not only the multiplicities are known, but the isoparametric hypersurfaces themselves.

In the case $p=4$, Cartan produced isoparametric hypersurfaces with four principal curvatures of
multiplicity one in $S^5$ and four principal curvatures of multiplicity two in $S^9$.  He noted all of his examples
are homogeneous, each being an orbit of a point under an appropriate closed subgroup of $SO(n+1)$.  Based on his results and the properties of his examples, Cartan asked the following three questions \cite{Car3}, all of which were answered
in the 1970's.

\begin{enumerate}\itemsep=0pt
\item  For each positive integer $p$, does there exist an isoparametric family with $p$ distinct principal
curvatures of the same multiplicity?

\item \looseness=-1 Does there exist an isoparametric family of hypersurfaces with more than three distinct~prin\-cipal curvatures such
that the principal curvatures do not all have the same multiplici\-ty?

\item\looseness=1 Does every isoparametric family of hypersurfaces admit a transitive group of isomet\-ries?
\end{enumerate}
These three questions will lead our note from now on, and we will reach to our main goal through answering to these questions, specially second question  is more interesting and solution to this one will guide us to find answer for our main goal.
 
M\"{u}nzner showed that each isoparametric hypersurface $M_t$ in the family separates
the sphere $S^n$
into two connected components $D_1$ and $D_2$, such that $D_1$ is a~disk bundle with f\/ibers of dimension
$m_1 + 1$ over $M_{+}$,
and $D_2$ is a disk bundle with f\/ibers of dimension $m_2 + 1$ over $M_{-}$, where $m_1$ and $m_2$ are the multiplicities
of the principal curvatures that give rise to the focal submanifolds $M_{+}$ and $M_{-}$, respectively.

This topological situation has been the basis
for many important results in this f\/ield concerning the number $p$ of distinct principal curvatures and the multiplicities
$m_1$ and $m_2$. In particular, in his second paper, M\"{u}nzner \cite{Mu2} assumed that $M$ is a compact,
connected embedded hypersurface that separates $S^n$ into two disk bundles
$D_1$ and $D_2$ over compact manifolds with f\/ibers of dimensions $m_1 + 1$ and $m_2 + 1$, respectively.  From this
hypothesis, M\"{u}nzner proved that the dimension of the cohomological ring $H^*(M,{\bf Z}_2)$ must be $2\alpha$, where
$\alpha$ is one of the numbers  1, 2, 3, 4 or 6.  He then proved that if $M$ is a compact, connected isoparametric hypersurface
with $p$ distinct principal curvatures, then $\dim H^*(M,{\bf Z}_2) = 2p$.
Combining these two results, M\"{u}nzner obtained his major
theorem that the number $p$ of distinct principal curvatures of an isoparametric hypersurface in a sphere $S^n$
must be  1, 2, 3, 4 or 6.

Hence this provides a negative answer for Cartan's first question.

In the next two coming sections we will discuss and provide the answer for the second and third Cartan's questions, respectively.

\section{Homogeneous isoparametric hypersurfaces}
\begin{dfn}
We say that isoparametric hypersurface is homogeneous if its isometry group acts on it transitively.
\end{dfn}
 
If $M^{n-1}\subset S^n=\mathrm{SO}(n+1)/\mathrm{SO}(n)\subset(\mathbb{R}^{n+1})$  is an orbit of a subgroup G of  $\mathrm{SO}(n+1)$, then $M^{n-1}$ is homogeneous and has constant principal curvature, so any orbits of such a group forms an isoparametric family in $S^n$.

A classification of homogenous isoparametric hypersurface families follow from a classification of all subgroups $G$ of $\mathrm{SO}(n+1)$ such that the principle orbits of $G$ in $S^n$ have codimension $1$ considered as a submanifold of $S^n$ or, what is the same thing, have codimension two as a submanifold of $\mathbb{R}^{n+1}$. We say that such an action of $G$ on $\mathbb{R}^{n+1}$ has cohomogeneity two. One can simplify the problem by asking for only those groups $G$ that are maximal in the sense, that there is no subgroup $H$ of $\mathrm{SO}(n+1)$ strictly containing $G$ and having the same orbit as $G$.
There is a list of all maximal isometric action with cohomogeneity two of a compact group G on Euclidean space $\mathbb{R}^{n+1}$ by Hsiang and Lawson, it turns out that this list consist exactly of the isotropy representations (action of Ad ) of rank two symmetric spaces.
Hence homogeneous isoparametric hypersurfaces in a sphere are  principle orbits of the isotropy representation of a rank two symmetric spaces. 
We explain it in more explicit way here:

Let $G/H=N^{n+1}$ be a symmetric space of compact type and consider its Cartan's decomposition to orthogonal symmetric Lie algebra:
\bdm
\mathfrak{g}=\mathfrak{h}\oplus\mathfrak{p},
\edm
where $\mathfrak{h}$ is Lie algebra of H and $\mathfrak{p}$ is a subspace of  $\mathfrak{g}$ which can be identified by
\bdm
T_xN=\mathfrak{p},
\edm
so $\dim\mathbf{N}=\dim\mathfrak{p}$.

Define unit sphere $S^n$ in $\mathfrak{p}$ with respect to the positive definite inner product in $\mathfrak{p}$ induced from the Killing form of $\mathfrak{g}$:
\[ S^n=\{u\in\mathfrak{p}|\hspace{2mm} \langle u,u\rangle=1\}. \]
Suppose $\mathfrak{a}=\mathfrak{g_0}$ is a maximal abelian subspace of $\mathfrak{p}$,
then we have a decomposition of $\mathfrak{g}$:
\bdm
\mathfrak{g}=\mathfrak{g_0}+\sum_{\alpha\in\Lambda}\mathfrak{g_\alpha}.
\edm
Let $\Delta$ be the set of all positive roots of $\mathfrak{g}$ with respect to $\mathfrak{a}$, and put:
\bdm
\Delta^{\ast}=\{\lambda\in\Delta|\hspace{2mm}  \frac{\lambda}{2}\notin\Delta\},
\edm
then Takagi and Takahashi  \cite {TT} proved that:
\begin{thm}\begin{enumerate}
\item Every orbit M in $\mathfrak{p}$  under K meets $\mathfrak{a}$.
\item An orbit M  has the highest dimension iff M contains no singular elements of $\mathfrak{a}$ and the highest dimension is exactly dim$\mathfrak {p}$-dim$\mathfrak {a}$.
(An element $X\in\mathfrak{a}$ is called a singular iff $\exists Y\in\mathfrak{g}\backslash\mathfrak{a}$  with [X,Y]=0, i.e $\exists\alpha\in\Lambda$ with $\alpha(X)=0$)
\item If the rank of $(G,H)$ is 2 then an orbit M of a unit vector in $S^n$ under K of 
 highest dimension is a hypersurface of $S^n$ and the principal curvature of M in $S^n$ 
are given by $-\frac{\lambda(B)}{\lambda(A)}$, where $\lambda\in\Delta^{\ast}$
and the pair ${A,B}$ is an orthonormal basis of $\mathfrak{a}$ such that $A\in M\bigcap\mathfrak{a}$.
The multiplicity of  $-\frac{\lambda(B)}{\lambda(A)}$ is equal to the sum of the multiplicity of  $\lambda$ and  $2\lambda$.
\end{enumerate}
\end{thm}
Cartan's list of  irreducible symmetric spaces of compact type of rank 2 and preceding theorem lead us to  the following table:
\vspace{2mm}

\begin{tabular}{|c|c|c|c|c|}
	\hline
$\mathfrak{g}$ &  $\mathfrak{h}$    &   $dim M$  & $p$   & $m_{i}$ \\
	\hline\hline
$\mathfrak{su}(3)$&$\mathfrak{so}(3)$ &3 &3& $m_i=1$\\\hline
$\mathfrak{su}(3)+\mathfrak{su}(3)$&$\mathfrak{su}(3)$ &6 & 3 &$m_i=2$\\\hline
$\mathfrak{su}(6)$&$\mathfrak{sp}(3)$ & 12&3 &$m_i=3$\\\hline
$\mathfrak{e_6}$&$\mathfrak{f_4}$&24 &3 &$m_i=3$\\\hline
$\mathfrak{so}(n+2),n\geq 3$&$\mathfrak{so}(n)+\mathfrak{so}(2)$ &2n-2 &4 &$m_1=m_3=1;$ $m_2=m_4=n-2$\\\hline
$\mathfrak{su}(n+2),n\geq 2$&$\mathfrak{su}(n)+\mathfrak{su}(2)$ & 4n-2&4 &$m_1=m_3=2;$ $m_2=m_4=2n-3$ \\\hline
$\mathfrak{sp}(n+2),n\geq 2$&$\mathfrak{sp}(n)+\mathfrak{sp}(2)$ &8n-2 &4 &$m_1=m_3=4;$ $m_2=m_4=4n-5$ \\\hline
$\mathfrak{so}(5)+\mathfrak{so}(5)$&$\mathfrak{so}(5)$ &8 &4 &$m_i=2$\\\hline
$\mathfrak{so}(10)$&$\mathfrak{u}(5)$ &18 &4 &$m_1=m_3=4 $;$m_2=m_4=5 $\\\hline
$\mathfrak{e_6}$&$\mathfrak{so}(10)+\mathbb{R}$ &30 &4 &$m_1=m_3=6$;$m_2=m_4=9$\\\hline
$\mathfrak{g_2}$&$\mathfrak{so}(4)$ &6 & 6&$m_i=1$\\\hline
$\mathfrak{g_2}+\mathfrak{g_2}$&$\mathfrak{g_2}$ &12 & 6 &$m_i=2$\\\hline
\hline
\end{tabular}

\vspace{3mm}
The first 4 rows of table prove  Pawel Nurowski's claims.

A quick look to the this table shows that there are  many isoparametric hypersurfaces with $4$ principal curvatures such that their multiplicities are not equal, this  provides a positive answer to   Cartan's second question cited in section 2.
\vspace{7mm}
\begin{exa}
Only $S^7$ and $S^{13}$ have all possible type of distinct principal curvatures. Here we want to see in  detail all isoparametric hypersurfaces of $S^7$. Let $x=(x_1,...,x_8)$ be the coordinate of $\mathbb{R}^8$ and $f=F\mid S^7$, denote $k$-dimensional sphere with radius $"r"$ by $ S^k(r)$ :
\vspace{3mm}
\begin{enumerate}
\item $\mathbf{p=1:}$  \hspace{3.5mm}$F(x)=x_8 $   ,\hspace{33mm} $f^{(-1)}(t)=S^6(\sqrt{1-t})$ ,\hspace{4mm} $ t\in(-1,1)$.

\vspace{3mm}
\item$\mathbf{p=2:}$\hspace{4mm}  $F(x)=\sum^{k}_{i=1}x^2_i-\sum^{8}_{j=k+1}x^2_j$ ,\hspace{3mm} $f^{(-1)}(t)=S^k\left(\sqrt{\frac{1-t}{2}}\right )\times S^{6-k}\left(\sqrt{\frac{1+t}{2}}\right)$.  
\vspace{3mm}
\item $\mathbf{p=3:}$  \bdm F(x)=u^3-3uv^2
+\frac{3}{2}u^2(|x|^2+|y|^2-2|z|^2)+\frac{3\sqrt{3}}{2}v(|x|^2-|y|^2)+\frac{3\sqrt{3}}{2}(xyz+\overline{x}\overline{y}\overline{z})\
\edm
where, $x=(u,v,x,y,z)\in\mathbb{R}^2\times\mathbb{C}^3=\mathbb{R}^8$,

and $f^{-1}(t)\cong\mathrm{SU}(3)/\textbf{T}^2$=isotropy orbit of $\mathrm{SU}(3)\mathrm{SU}(3)/\mathrm{SU}(3)$. We may emphasize that this case is already our "good" case and we knew this from section $2$, and polynomial $F(x)$ corresponds to the second row in the table, it is expression of tensor $\Upsilon$ in  $\mathrm{SU}(3)\mathrm{SU}(3)/\mathrm{SU}(3)$ of dimension $8$. This polynomial of degree $3$ can be modified  by considering  right order of (x,y,z) in $\mathbb{H}$ and $\mathbb{O}$ in order to express the tensor $\Upsilon$ by polynomial of degree $3$ in $SU(6)/Sp(3)$ and $E_6/F_4$ respectively.
\vspace{4mm}
\item$\mathbf{p=4:}$ \hspace{3mm} $F(x)$ is a homogeneous polynomial of degree 4, and $f^{-1}(t)$=isotropy orbit of $\mathrm{SO}(6)/\mathrm{SO}(2)\mathrm{SO}(4)$.

\vspace{3mm}
\item $\mathbf{p=6:}$  \hspace{5mm}$F(x)$ is a  homogeneous polynomial of degree 6, and $f^{-1}(t)\cong\mathrm{SO}(4)/\mathbb{Z}_2$=isotropy orbit of $\mathrm{G}_2/\mathrm{SO}(4)$.
\end{enumerate}
\end{exa}   

Proof of the theorem 3.1 shows that the normal great circles to the family of isoparametric hypersurface are the intersection of Cartan subalgebra with $S^n$. The focal points on a normal great circles are the singular elements in the corresponding Cartan subalgebra that lies in $S^n$, and the number $p$ of distinct principal curvatures relates to the Weyl group of the symmetric space, since the angle between the rays in the boundary of Weyl chamber is $\pi/p$. Order of the Weyl group is $2p$. Hence the principal curvatures and their multiplicities can be calculated from the roots of the symmetric space.

To give a concrete example take the oriented Grassmanian of two planes $\textbf{G}_2(\mathbb{R}^{n+3})$. Its rank is equal to two and the order of the Weyl group is eight. We therefore get an isoparametric family with $p=4$ in $S^{2n+1}$. Its multiplicities can be seen to be $1$ and $n-1$.

There is a short note of Nomizu[4] from the year 1973 in which he expressed the above example in terms of an isoparametric function on $S^{2n+1}$:
\bdm
F(x,y)=(|x|^2-|y|^2)^2+4\langle x,y\rangle ,
\edm

where  $(x,y)\in\mathbb{R}^{n+1}\times\mathbb{R}^{n+1}=\mathbb{R}^{2n+2}$.

\begin{exa}
Let $N^5={\mathrm{SU}(3)}/{\mathrm{SO}(3)}$, then $\mathfrak{su}(3)=\mathfrak{so}(3)\oplus\mathfrak{p}$, where $\mathfrak{p}$ is a vector subspace of $\mathfrak{su}(3)$, more precisely $\mathfrak{p}=\{X\in\mathfrak{su}(3)|X^t=X$ which can be identified canonically by $T_x{N^5}$, hence it is $5$-dimensional subspace. Consider the orbit 
\bdm
M:=\{kxk^{-1}| k\in\mathrm{SO}(3), \hspace{2mm}x\in\mathfrak{p}\}.
\edm
For $X\in\mathfrak{so}(3)$ we have $[X,x]=0$ iff $X=0$. Hence the map
\bdm
\phi:\mathrm{SO}(3)\rightarrow M
\edm
given by  $k\rightarrow Ad(k)x$ is local diffeomorphism, moreover $\mathrm{SO}(3)$ is compact therefore $\phi$ is really covering map. $\mathrm{SO}(3)$ is three dimensionalmanifold hence dim$M=3$, and $M$ is a homogeneous and therefore it is isoparametric hypersurface of $S^4$.
The tangent space of $M$ at $x$ is
\bdm
T_xM=[\mathfrak{so}(3),x].
\edm
The geodesic in $M$ with initial vector $[X,x]$ at $x$ is given by $\exp{(tX)}x\exp{(-tX)}$. Therefore the shape operator at $x$ with respect to a unit normal vector $\xi$ is given by
\bdm
A[X,x]=-[X,\xi].
\edm
Using
\bdm
\xi:=\frac{1}{6}\left(\begin{array}{ccc}
i&0&0\\
0&i&0\\
0&0&-2i\\
\end{array}\right),
\edm
 computation shows that the principal curvatures are $+\frac{1}{\sqrt{3}},-\frac{1}{\sqrt{3}} $ and $0$. Hence $M$ has $p=3$ distinct principal curvatures each of multiplicity $1$.
\end{exa}
\section{Inhomogeneous isoparametric Hypersurfaces}\noindent
A systematic approach to find inhomogeneous isoparametric hypersurface was given by Ferus, Karcher and M\"{u}nzner in \cite{FKM} who associated to a representation of Clifford algebra $\mathbf{C}_{m-1}$ on $\mathbb{R}^l$ a Cartan-M\"{u}nzner polynomial that gives rise to an isoparametric hypersurface with $g=4$ principal curvatures in $S^{2l-1}$.

\vspace{2mm}
\begin{dfn}
For each integer $m\geq 0$, the Clifford algebra $C_m$ is the associative algebra over R, that is generated by a unity 1 and the element $e_1,...,e_m$ subjected only to the relations
\bdm
e_i^2=1  , \hspace{10mm} e_ie_j=e_je_i,
\edm  
where  $ i\neq j$  , and   $1\leq i\neq j\leq m$.
\end{dfn}

A representation of $C_{m-1}$ on $\mathbb{R}^l$ corresponds to a set of skew-symmetric matrices $E_1,...,E_{m-1}$ in the orthogonal group $\mathrm{O}(l)$ such that
\bdm
E_iE_j+E_jE_i=-2\delta_{ij}Id.
\edm

\begin{dfn}
An $(n+1)$-tuple $(P_0,...,P_m)$ of symmetric endomorphism of $\mathbb{R}^{2l}$ is called a Clifford system if
\bdm
P_iP_j+P_jP_i=2\delta{ij}Id.
\edm
\end{dfn}
We define $(P_0,...,P_m)$ of symmetric endomorphism on $\mathbb{R}^{2l}$ by
\bdm
P_0=(x,y) ,\hspace{5mm} P_1=(y,x) ,\hspace{5mm} P_{i+1}=(E_iy,-E_ix)  ;\hspace{10mm}  x,y\in\mathbb{R}^{l}.
\edm

\begin{dfn}
A representation of Clifford system is irreducible iff representation of $C_{m-1}$ is irreducible. 
\end{dfn}

Therefore by Atiyah, Bott and Shapiro [3] Clifford algebra $C_{m-1}$ has an irreducible representation of degree $l$ if and only if $l=\delta(m)$ as in the table:

\vspace{5mm}

\bdm
\begin{tabular}{ccc}
	\hline
 $ m $ & $\mathbf{C}_{m-1} $  & $\delta({m})$ \\
	\hline
1&$\mathbb{R}$ &1\\
2&$\mathbb{C}$& 2\\
3&$\mathbb{H}$& 4\\
4&$\mathbb{H\oplus H}$&4\\
5&$\mathbb{H}(2)$&8\\
6&$\mathbb{C}(4)$&8\\
7&$\mathbb{R}(8)$&8\\
8&$\mathbb{R}(8)\oplus\mathbb{R}(8)$ &8\\
k+8&$\mathbb{C}_{k-1}(16)$&$16\delta(k)$\\
\end{tabular}
\edm

\vspace{6mm}

Reducible representations of $C_{m-1}$  on $\mathbb{R}^l$ can be obtained by taking a direct sum of $k$ irreducible representation of $C_{m-1}$ on $\mathbb{R}^{\delta(m)}$ for $l=k\delta(m)$ , $k>1$.

\begin{thm}
Given a Clifford system $(P_0,...,P_m)$ on $\mathbb{R}^{2l}$ such that $m_1:=m$ and $m_2$:=l-m-1 are positive, then $F:\mathbb{R}^{2l}\ra\mathbb{R}$
\bdm
F(x):=\langle x,x\rangle^2-2\sum^m_{i=1}\langle P_ix,x\rangle^2
\edm
is an isoparametric function defining an isoparametric family with $p=4$, and multiplicities $(m_1,m_2)$.
\end{thm}
\begin{proof}
We need to check that this polynomial satisfies in Cartan-M\"{u}nzner differential equations (\ref{eq:C-M}):
\bdm
grad\left(\mathbf{F}(x)\right)= 4\langle x,x\rangle x-8\sum_i\langle P_ix,x\rangle P_ix ,
\edm
hence
\begin{eqnarray*}
|grad(\mathbf{F}(x))|^2 &=& 16\langle x,x\rangle^3+64\sum_{i,j}\langle 
P_ix,x\rangle\langle P_ix,P_jx\rangle-64\sum_{i}\langle P_ix,x\rangle^2\langle x,x\rangle=16\langle x,x\rangle^3, 
\end{eqnarray*}
because $\langle P_ix,P_jx\rangle\langle P_jP_ix,x\rangle=\langle x,x\rangle\delta_{ij}$.
Note that $P_jP_i$ is skew-symmetric for $i\neq j$. Furthermore,
\bdm
\Delta\textbf{F}(x)=4(2l+2)\langle x,x\rangle-2\sum_i\left(2<\langle P_ix,x\rangle\Delta(\langle P_x,x\rangle)+2|{grad}\langle P_ix,x\rangle|^2\right)=8(m_2-m_1)\langle x,x\rangle,
\edm
and this completes the proof.
\end{proof}

Due to Ferus, Karcher and M\"{u}nzner, isoparametric hypersurfaces resulting from this theorem are called isoparametric hypersurfaces of FKM-type .

We present a list of the low-dimensional Clifford examples with $(m_1,m_2)=(m,k\delta(m)-m-1)$.
In the table below, the case where $m\leq 0$ are denoted by a dash.

\vspace{5mm}
$$
\begin{tabular}{|c|cccccccccc}
	\hline
$ \delta (m)$ &$1$& $2$ & $4$ &$4$&$8$ &$8$ &$8$ &$8$ & $16$ &...\\
	\hline
k&&&&&&&&&&\\
1&-&-&-&-&(5,2)&(6,1)&-&-&(9,6)&...\\
2&-&(2,1)&(3,4)&(4,3)&(5,10)&(6,9)&(7,8)&(8,7)&(9,22)&...\\
3&(1,1)&(2,3)&(3,8)&(4,7)&(5,18)&(6,17)&(7,16)&(8,15)&(9,38)&...\\
4&(1,2)&(2,5)&(3,12)&(4,11)&(5,26)&(6,25)&(7,24)&(8,23)&(9,54)&...\\
5&(1,3)&(2,7)&(3,16)&(4,17)&(5,34)&(6,33)&(7,32)&(8,31)&(.)&...\\
.&.&.&.&.&.&.&.&.&.&....\\
.&.&.&.&.&.&.&.&.&.&...\\
.&.&.&.&.&.&.&.&.&.&...\\

\end{tabular}
$$

\vspace{5mm}

This table not only contains all homogeneous isoparametric hypersurfaces from section 3 with exception of two homogeneous families, having multiplicities $(2,2)$ and $(4,5)$ but also many inhomogeneous isoparametric hypersurfaces, because their multiplicities do not agree with the multiplicities of any homogeneous isoparametric hypersurface from the table of previous section. 
However Ferus, Karcher and M\"{u}nzner also gave a geometric proof of the inhomogeneity of many of their examples,  more precisely  they proved following theorem: 

\begin{thm}\cite{FKM}.
Suppose $3\leq 3m_1\leq m_2+9$, and in case m=4, $P_0,P_1,P_2,P_3\neq\pm Id$. 
Then the isoparametric family is inhomogeneous.
\end{thm}

This theorem says that there are many isoparametric hypersurfaces which do not admit isometry group such that act  transitively on the family of isoparametric hypersyrface. Hence this gives negative answer to the third question of Cartan. 

\textbf{Open Questions: }
\begin{enumerate}
\item All known examples of isoparametric hypersurfaces with four distinct principal curvatures are either FKM-type or two exceptional homogeneous families, having multiplicities $(2,2)$ and $(4,5)$. But one still does not know whether further examples exist.
\vspace{3mm}
\item If $p=6$, then Abresch \cite{Ab} proved that $m_1=m_2$ and there are only the possibilities $m_1=m_2=1$ , $m_1=m_2=2$ and we saw in previous section that there are homogeneous example in both cases. In the case of multiplicity m=1, Dorfmeister and Neher \cite{DN} showed that an isoparametric hypersurface must be homogeneous, thereby completely classifying that case. The proof of Dorfmeister and Neher is quite algebraic in nature.
The classification of isoparametric hypersurfaces with six principal curvatures of multiplicities $m=2$ is part of problem 34 on Yau's \cite{Yau} list of important open problems in geometry, and remains an open problem.

\end{enumerate}

    
\end{document}